\documentclass[11pt]  {amsart} 
\usepackage{amsmath,latexsym,amssymb,amsmath,
amscd,amsthm,amsxtra}

\newtheorem{theorem}{Theorem}[section]
\newtheorem{proposition}{Proposition}[section]
\newtheorem{lemma}{Lemma}[section]
\newtheorem{corollary}{Corollary}[section]

\newtheorem{remark}{\textbf{Remark}}[section]

\def\NN{\mathbb{N}}

\def\o{\omega}
\def\p{\partial}

\def\k{\kappa}
\def\l{\lambda}
\def\L{\Lambda}
\def\s{\sigma}

\def\S{\Sigma}
\def\RR{\mathbb R}\def\R{\mathbb R}
\def\KK{\mathcal K}
\def\HH{\mathbb H}
\def\H{\mathbb H}
\def\SS{{\mathbb S}}
\def\<{\langle}
\def\>{\rangle}

\def\H{{\mathbb H}}

\numberwithin{equation}{section}

\def \ds{\displaystyle}
\def \vs{\vspace*{0.1cm}}

\begin{document}

\title[Isoperimetric type problems in hyperbolic space]{Isoperimetric type problems and 
Alexandrov-Fenchel type inequalities in  the hyperbolic space}
\author{Guofang Wang}
\address{ Albert-Ludwigs-Universit\"at Freiburg,
Mathematisches Institut,
Eckerstr. 1,
79104 Freiburg, Germany}
\email{guofang.wang@math.uni-freiburg.de}
\thanks{GW is partly supported by SFB/TR71 ``Geometric partial differential equations'' of DFG}
\author{Chao Xia}\address{Max-Planck-Institut f\"ur Mathematik in den Naturwissenschaft, Inselstr. 22, D-04103, Leipzig, Germany}
\thanks{Part of this work was done while CX was visiting the mathematical institute of Albert-Ludwigs-Universit\"at Freiburg. He would like to thank the institute for its hospitality.
}
\email{chao.xia@mis.mpg.de}

\date{}


\begin{abstract}{ 
In this paper, we solve  various isoperimetric  problems for the quermassintegrals and the curvature integrals in the 
hyperbolic space $\H^n$, by using quermassintegral preserving curvature flows. As a byproduct, we obtain hyperbolic Alexandrov-Fenchel inequalities.
}

\end{abstract}
\

\medskip

\maketitle

\section{Introduction}

 Isoperimetric type problems play an important role in mathematics.
The classical isoperimetric theorem in the Euclidean space says that among all  bounded domains in $\RR^n$ with given volume, 
the minimum of the area of the boundary is achieved precisely by the round balls. This can be formulated as an optimal inequality
\begin{eqnarray}\label{isop}
\hbox{Area}(\p K)\geq n^\frac{n-1}{n}\o_{n-1}^\frac1n\hbox{Vol}(K)^\frac{n-1}{n}, 
\end{eqnarray}
for any bounded domain $K\subset\RR^n$,
and equality holds if and only if $K$ is a geodesic ball. Here and throughout this paper,  $\o_k$ denotes the Lebesgue measure of $k$-dimensional unit sphere $\SS^k$, and by a bounded domain we mean a compact set with non-empty interior.
When $n=2$, inequality \eqref{isop} is 
\begin{eqnarray}\label{isop0}
 L^2\ge 4\pi A,\end{eqnarray}
where $L$ is the length of a  closed curve $\gamma$ in $\R^2$ and $A$ is the area of the enclosed domain by $\gamma$.
\eqref{isop} and \eqref{isop0}
are  the classical isoperimetric inequalities. Their general forms are  the Alexandrov-Fenchel quermassintegral inequalities.
A special, but interesting class of the Alexandrov-Fenchel quermassintegral establishes
the relationship between the quermassintegrals or the curvature integrals:
\begin{eqnarray}\label{AF}
\int_{\p K} H_k d\mu\geq \o_{n-1}^\frac{k-l}{n-1-l}\left(\int_{\p K} H_l d\mu\right)^\frac{n-1-k}{n-1-l}, \quad 0\leq  l<k\leq n-1,
\end{eqnarray}
for any  convex bounded domain $K\subset \RR^n$  with $C^2$ boundary, where $H_k$ are the \textit{(normalized) $k$-th mean curvature} of $\p K$ as an embedding in $\RR^n$.
These inequalities have been intensively studied by many mathematicians and have many applications in differential geometry and integral geometry. 
See the excellent books of Burago-Zalgaller  \cite{BuragoZalgaller}, Santalo \cite{San} and Schneider \cite{Schneider}.
Recently, the Alexandrov-Fenchel quermassintegral inequalities in $\R^n$ have been extended to certain classes of  non-convex domains. See for example \cite{ChangWang, GL, Huisken}.

All these above inequalities solve the problem if one geometric quantity attains its minimum or maximum at geodesic balls
 among  a class of (smooth) bounded  domains in $\RR^n$ with another  given geometric quantity. We call such problems {\it  isoperimetric type problems}.

It is a very natural question to ask if such  isoperimetric type problems also hold in the hyperbolic space $\H^n$. We remark that in this paper $\H^n$ denotes the hyperbolic space with  the sectional curvature $-1$. One of main motivations to study this problem comes naturally from 
integral geometry in $\H^n$. Another main motivation comes from the recent study of ADM mass, Gauss-Bonnet-Chern mass and quasi-local mass in asymptotically hyperbolic manifolds.
The isoperimetric problem between volume and area in $\H^n$ was already solved by Schmidt \cite{Schmidt} 70 years ago.
 Due to its complication, a simple explicit  inequality like  \eqref{isop} is in general not available.  When  $n=2$, there is an  explicit form, namely 
 the hyperbolic isoperimetric inequality in this case is
\begin{eqnarray}
\label{isoH} L^2\ge 4\pi A+A^2,\end{eqnarray}
where $L$ is the length of a closed curve $\gamma$ in $\H^2$ and $A$ is the area of the enclosed domain by $\gamma$. Moreover, equality holds if and only if $\gamma$ is a circle.
Comparing to \eqref{isop0}, inequality \eqref{isoH} has an extra term. This is a well-known phenomenon, which indicates that the  isoperimetric type problems
in $\H^n$ 
are more complicated than the ones in $\R^n$.

Till now, the Alexandrov-Fenchel type inequalities or the isoperimetric type problems in the hyperbolic space are quite open except some special cases. See for example 
\cite{BM1,  GS, GHS}. In \cite{BM1} and \cite{GS}, some interesting inequalities between curvature integrals and quermassintegrals have been obtained. However, the results
obtained there are far away from being optimal. Here we say that a geometric inequality for bounded domains is optimal, if equality holds if and only if the domain
is a geodesic ball. In other words, only geodesic balls solve the corresponding isoperimetric problem.
More  recently, several interesting  works have appeared in this research field, see \cite{BHW, dLG,  GWW, GWW2, LWX}. 
In \cite{GWW, GWW2, LWX}, the authors solve some special cases of the isoperimetric type problems by establishing 
the following inequalities as the Alexandrov-Fenchel inequalities \eqref{AF} for the curvature integrals. For $2\leq 2k\le n-1$,
\begin{equation}\label{AFk}
\begin{array}{rcl}
\ds \int_{\p K} H_{2k}d\mu\ge \ds\vs \omega_{n-1}\left\{ \left( \frac{|\p K|}{\omega_{n-1}} \right)^\frac 1k +
\left( \frac{|\p K|}{\omega_{n-1}} \right)^{\frac 1k\frac {n-1-2k}{n-1}} \right\}^k,
\end{array}
\end{equation}
for any horospherical convex domain $K\subset \H^n$. Here $|\p K|$ is the area of $\p K$. This is optimal, in the sense that equality holds if and only if $K$ is
a geodesic ball in $\H^n$. When $k=1$, inequality \eqref{AFk} was proved in \cite{LWX} under a weaker condition that $\p K$ is star-shaped and 2-convex.

In order to state our results we give more precise definitions about quermassintegrals and curvature integrals.

Let us first recall two different kinds of convexity in $\HH^n$.
A domain $K\subset\HH^n$ is said to be \textit{(geodesically) convex} if for every point $p\in\p K$, $K$ is contained in the enclosed ball of some totally geodesic sphere through $p$. 
A domain $K\subset \HH^n$ is said to be \textit{horospherical convex}, or \textit{h-convex}, or have \textit{h-convex boundary}, if for every point $p\in\p K$, $K$ is contained in the enclosed ball of some horosphere $S_h(p)$ through $p$. Moreover, it is said to be \textit{strictly h-convex} if  $\p K\cap S_h(p)=p$. Recall that a horosphere in $\HH^n$ is a hypersurface obtained
as the limit of a geodesic sphere of $\HH^n$ when its center goes to the infinity along
a fixed geodesic ray. The (strict) h-convexity of $K\subset\HH^n$ is equivalent to that all the principal curvatures of its boundary $\p K$ are (strictly) bounded below by $1$. The geodesic balls in $\HH^n$ are all strictly h-convex. An h-convex domain must be convex, but the converse is not true. In some sense, the horospherical convexity is more natural geometric concept than the convexity in $\H^n$, see for example \cite{GNS}. The horospherical convexity plays a crucial  role in  the proof of \eqref{AFk} in \cite{GWW, GWW2} for $k\ge 2$. It is also crucial
 for this paper.


For a (geodesically) convex domain $K\subset \HH^n$, the \textit{quermassintegrals} are defined by
\begin{eqnarray*}
&&W_k(K):=\frac{(n-k)\o_{k-1}\cdots\o_0}{n\o_{n-2}\cdots\o_{n-k-1}}\int_{\mathcal{L}_k}\chi(L_k\cap K)dL_k, \quad k=1,\cdots,n-1;
\end{eqnarray*}
where $\mathcal{L}_k$ is the space of $k$-dimensional totally geodesic subspaces $L_k$ in $\HH^n$ and $dL_k$ is the natural (invariant) measure on $\mathcal{L}_k$. The function $\chi$ is given by $\chi(K)=1$  if $K\neq \emptyset$ and $\chi(\emptyset)=0.$
For simplicity, we also use the convention 
\begin{eqnarray*}
&& W_0(K)=\hbox{Vol}(K), \quad W_n(K)=\frac{\o_{n-1}}{n}.
\end{eqnarray*}
Remark that by definition we know
\[ W_1(K)=\frac1n |\p K|.\]
If the boundary $\p K$ is $C^2$-differentiable, the \textit{curvature integrals} are defined by
\begin{eqnarray*}
V_{n-1-k}(K)=\int_{\p K} H_k d\mu, \quad k=0,\cdots, n-1,
\end{eqnarray*}
where $H_k$ are the (normalized) $k$-th mean curvature of $\p K$ as an embedding in $\H^n$ and $d\mu$ is the area element on $\p K$ induced from $\HH^n$.

 From the viewpoint of integral geometry, the quermassintegrals seem to be more important and play a central role. 
Nevertheless, the curvature integrals are also very important geometric quantities not only in integral geometry, but also in the theory of submanifolds.
In $\RR^n$, the quermassintegrals coincide the curvature integrals, up to a constant multiple. However, 
the quermassintgrals and the curvature integrals in $\HH^n$ do not coincide. Nevertheless they  are closely related (see e.g. \cite{Sol2}, Proposition 7):
\begin{eqnarray*}
V_{n-1-k}(K)&=& n\left( W_{k+1}(K)+\frac{k}{n-k+1}W_{k-1}(K)\right),\quad k=1,\cdots,n-1,\\
V_{n-1}(K)&=& n W_1(K)=|\p K|.\nonumber
\end{eqnarray*}

 In this paper, we will solve a large class of  the isoperimetric type problems in $\H^n$ 
 involving the quermassintegrals and the curvature integrals for  h-convex bounded domains with smooth boundary.

 The first main result of this paper is the following Alexandrov-Fenchel type inequalities for the quermassintegrals.

\begin{theorem}\label{thm1} Let $\KK$ be the space of  h-convex bounded domains in $\HH^n$ with smooth boundary and $K\in \KK$.  For $0\leq l<k\leq n-1$, we have 
\begin{eqnarray*}
W_k(K)\geq f_k\circ f_l^{-1}(W_l(K)).
\end{eqnarray*}
Equality  holds if and only if $K$ is a geodesic ball. Here $f_k:[0,\infty)\to \RR_+$ is a monotone function defined by $f_k(r)=W_k(B_r)$, the $k$-th quermassintegral for  the geodesic ball of radius $r$, and $f_l^{-1}$ is the inverse function of $f_l$.
In other words, the minimum of $W_k$ among the domains in $\KK$ with given $W_l$ is achieved precisely by  geodesic balls.
\end{theorem}

 Moreover, from Theorem \ref{thm1}  we solve the following isoperimetric type problems.
  
\begin{theorem}\label{thm2} Let $\KK$ be the space of  h-convex bounded domains in $\HH^n$ with smooth boundary.  Then the following holds:
\begin{itemize}
\item[(i)] For $0\le l<  k\leq n-1$, $V_{n-1-k}$ 
attains its minimum at a geodesic ball among the domains in $\KK$ with given $W_{l}$;
\item[(ii)] For $0\leq k\leq n-1$, $V_{n-1-k}$ attains its minimum at a geodesic ball among the domains in $\KK$ with given volume $W_0=Vol$;
\item[(iii)] For $1\leq k\leq n-1$, $V_{n-1-k}$ attains its minimum at a geodesic ball among the domains in $\KK$ with given area $|\partial K|=nW_1=V_{n-1}$ of the boundary $\partial K$;
\item[(iv)] For $0\leq l<k \leq {n-1}$ and $k-l=2m$ for some $m\in\NN$, $V_{n-1-k} $ attains its minimum at a geodesic ball among the domains in $\KK$ with given $V_{n-1-l}$.
\end{itemize}
\end{theorem} 

Theorem \ref{thm1} and \ref{thm2} give an affirmative   answer 
to the question posed by Gao-Hug-Schneider in \cite{GHS} for $\H^n$ (in the case of  h-convex bounded domains with smooth boundary).

\

Unlike in $\R^n$, most of above results for  quermassintegrals and 
the curvature integrals in $\H^n$ have no explicit (inequality) form. As mentioned above, even the classical 
isoperimetric problem between volume and area in $\H^n$ solved in \cite{Schmidt} has in general no explicit from.
Here we are able to formulate Statement (iii) in Theorem \ref{thm2}
in an optimal inequality.
\begin{theorem}\label{thm3} Let $1\le k\le n-1$.  Any  h-convex bounded domain $K$ in $\H^n$ with smooth boundary
satisfies
\begin{equation}\label{eq_add}
\int_{\p K}H_{k}d\mu\geq \omega_{n-1}\bigg\{\bigg(\frac{|\p K|}{\omega_{n-1}}\bigg)^{\frac{2}{k}}+\bigg(\frac{|\p K|}{\omega_{n-1}}\bigg)^{\frac{2}{k}\frac{(n-k-1)}{n-1}}\bigg\}^{\frac {k}{2}}.
\end{equation} Equality holds if and only if $K$ is a geodesic ball.
\end{theorem}
Inequality \eqref{eq_add} was called as a hyperbolic Alexandrov-Fenchel inequality in \cite{GWW}.
As mentioned above, \eqref{eq_add} was proved in \cite{LWX} for $k=2$ under a weaker condition, 
in \cite{GWW} for $k=4$ and in \cite{GWW2} for general even $k$. For general odd integer $k$ inequality \eqref{eq_add} was conjectured in
\cite{GWW2} after the authors showed \eqref{eq_add} for $k=1$ with a help of a result of Cheng and Xu \cite{CZ}. For the related work about 
the result of Cheng and Xu \cite{CZ},
see also \cite{DT}, \cite{GW} and \cite{GWX}. 
 
 Recently Theorem \ref{thm3} (for $k$ odd) was used in \cite{GWW_AH} to prove a Penrose type inequality for a higher order mass on asymptotically hyperbolic manifolds.

The approaches  used in \cite{GWW, GWW2, LWX}, and also in \cite{BHW, dLG}, are finding a suitable
geometric quantity, which is monotone under a suitable inverse curvature flow studied by Gerhardt \cite{Gerhardt},
and managing to compute the limit of the  geometric quantity.
However in this paper we will not use an inverse curvature flow. Instead
we will use a (normalized) generalized mean curvature flow to prove Theorem \ref{thm1}. 
The crucial points of this paper are: (i) the choice of  the quermassintegrals $W_k$ as this suitable geometric quantity,
(ii) the use of
 the quermassintegral preserving curvature flows, along which one quermassintegral is preserved and the other is monotone. 
 The flow we consider is
 \begin{equation}\label{main_flow}
 \frac{\partial X}{\partial t}(x,t)=\left\{c(t)-\left(\frac{H_k}{H_l}\right)^{\frac{1}{k-l}}(x,t)\right\}\nu(x,t),
 \end{equation}
where $\nu(\cdot,t)$ is the outer normal of the evolved hypersurface and $c(t)$ is defined by 
\[ c(t)=c_l(t)=\frac{\int_{\S_t} H_k^{\frac{1}{k-l}}H_l^{1-\frac{1}{k-l}}d\mu_t}{\int_{\S_t} H_l d\mu_t}.
\]
We will show that this flow converges exponentially to a geodesic sphere, provided that the initial hypersurface is h-convex.
The study of this flow is motivated by the work of \cite{CM, GL2, Ma, Mc}, who considered the mixed volume (in our words, the curvature integrals) preserving curvature flows in $\RR^n$ and $\HH^n$ respectively. In \cite{Ma} the  isoperimeteric result of Schmidt mentioned above was reproved by a flow method.
The method of using geometric flows to prove geometric inequalities seems to be powerful. Various flows have been employed to prove geometric inequalities, 
see for instance \cite{An2, BHW, dLG, GWW, GWW2, GWX, GL, GL2, GuanW, LWX, Ma,Mc, Schulze}. 

\

The rest of this paper is organized as follows. In Section 2, we present basic concepts and facts about  integral geometry in the hyperbolic space. In Section 3, we study the quermassintegral preserving curvature flows and prove a rigidity result. In Section 4, we choose a special flow to prove our main theorems.

\

\section{Curvature integrals and Quermassintegrals}

In this section, we recall some basic concepts in integral geometry in the hyperbolic space, we refer to Santal\'o's book \cite{San}, Part IV, and Solanes' thesis \cite{Sol} for more details.

For a (geodesically) convex domain $K\subset \HH^n$, the \textit{quermassintegrals} are defined by
\begin{eqnarray}\label{quer}
&&W_k(K):=\frac{(n-k)\o_{k-1}\cdots\o_0}{n\o_{n-2}\cdots\o_{n-k-1}}\int_{\mathcal{L}_k}\chi(L_k\cap K)dL_k, \quad k=1,\cdots,n-1;
\end{eqnarray}
where $\mathcal{L}_k$ is the space of $k$-dimensional totally geodesic subspaces $L_k$ in $\HH^n$ and $dL_k$ is the natural (invariant) measure on $\mathcal{L}_k$. The function $\chi$ is given by $\chi(K)=1$  if $K\neq \emptyset$ and $\chi(\emptyset)=0.$
For simplicity, we also use the notation \begin{eqnarray*}
&& W_0(K)=\hbox{Vol}(K), \quad W_n(K)=\frac{\o_{n-1}}{n}.
\end{eqnarray*}
It is clear from  definition \eqref{quer}  that the quermassintegrals $W_k$, $k=0,1,\cdots,n-1$,  are strictly increasing under set inclusion, i.e., 
\begin{eqnarray}\label{inclus}
\hbox{if }K_1\subsetneqq K_2,\hbox{ then }W_k(K_1)< W_k(K_2).
\end{eqnarray}
This simple fact plays a role in the proof of the convergence of curvature flows considered below.

\

Let $\sigma_k$ be the $k$-th elementary symmetric function $\sigma_k: \RR^{n-1}\to \RR$  defined by 
\begin{eqnarray*}
\sigma_k(\L)=\sum_{i_1<\cdots<i_k} \l_{i_1}\cdots\l_{i_k} \quad \hbox{ for } \L=(\l_1,\cdots,\l_{n-1})\in \RR^{n-1}.
\end{eqnarray*}
As convention, we take $\s_0=1$.
The definition of $\sigma_k$ can be easily extended to the set of all symmetric matrix. The Garding cone $\Gamma_k^+$ is defined as 
\begin{eqnarray*}
\Gamma_k^+=\{\L\in \RR|\sigma_j(\L)>0,\quad \forall j\leq k\}.
\end{eqnarray*} We denote by $ \overline{\Gamma_k^+}$  the closure of $\Gamma_k^+$. 

Let $H_k=H_k(\L)=\frac{\s_k(\L)}{C^{k}_{n-1}}$ the normalized symmetric functions. We have the following Newton-MacLaurin inequalities.
For the proof we refer to a survey of Guan \cite{Guan}.
\begin{proposition}
For $1\leq l<k\leq n-1$ and $\L\in \overline{\Gamma_k^+}$, the following inequalities hold:\begin{eqnarray}\label{N1}
H_{k-1} H_{l}\geq H_kH_{l-1}.
\end{eqnarray}
\begin{eqnarray}\label{N2}
H_{l}\geq H_k^{\frac{l}{k}}.
\end{eqnarray}
Equalities hold in \eqref{N1} or \eqref{N2} if and only if $\l_i=\l_j$ for all $1\leq i,j\leq n-1$.
\end{proposition}

\

For a  domain $K\subset \HH^n$, if the boundary $\p K$ is $C^2$-differentiable,  the \textit{(normalized) $k$-th mean curvatures} are
\begin{eqnarray*}
H_k(x)=H_k(\k(x))\quad \hbox{ for }x\in \p K,\quad k=0,\cdots,n-1,
\end{eqnarray*}
where $\k=(\k_1,\cdots,\k_{n-1})$ is the set of  the principal curvatures of $\p K$ as an embedding in $\HH^n$.
The \textit{curvature integrals} are defined by
\begin{eqnarray*}
V_{n-1-k}(K)=\int_{\p K} H_k d\mu, \quad k=0,\cdots, n-1,
\end{eqnarray*}
where $d\mu$ is the area element on $\p K$ induced from $\HH^n$.

The curvature integrals have a similar meaning of the mixed volume in the Euclidean space, in view of the Steiner formula (see \cite{San},  IV.18.4) 
which says that for a smooth convex domain $K$ and some positive number $\rho\in \RR$, its parallel set $K[\rho]:=\{x\in \HH^n| d_{\HH^n}(x,K)\leq \rho\}$ has the volume
\begin{eqnarray*}
\hbox{Vol}(K[\rho])=\hbox{Vol}(K)+\sum_{k=0}^{n-1} C_{n-1}^{k}V_k(K)\int_0^\rho \cosh^k (s)\sinh^{n-1-k}(s)ds.
\end{eqnarray*}

Recall that the quermassintgrals and the curvature integrals are related (see e.g. \cite{Sol2}, Proposition 7) by
\begin{eqnarray}\label{relation}
V_{n-1-k}(K)&=& n\left( W_{k+1}(K)+\frac{k}{n-k+1}W_{k-1}(K)\right),\quad k=1,\cdots,n-1,\\ V_{n-1}(K)&=&n W_1(K)=|\p K|.\nonumber
\end{eqnarray}
 From \eqref{relation} it is easy to express $W_k$ as a linear combination of several curvature integrals (see  \cite{San}, IV.17.4, \cite{Sol2}, Corollary 8):

\begin{itemize}\item for $1\leq k\leq n-1$ and $k$ is  even,
\begin{eqnarray}\label{even}
W_k(K)&=&\frac1n \sum_{i=0}^{\frac{k}{2}-1} (-1)^i\frac{(k-1)!!(n-k)!!}{(k-1-2i)!!(n-k+2i)!!}\int_{\p K} H_{k-1-2i}d\mu\nonumber\\&&+(-1)^{\frac{k}{2}}\frac{(k-1)!!(n-k)!!}{n!!}\hbox{Vol}(K);
\end{eqnarray}
\item for $1\leq k\leq n-1$ and $k$ is  odd,
\begin{eqnarray}\label{odd}
&& W_k(K)=\frac1n \sum_{i=0}^{\frac{k-1}{2}} (-1)^i\frac{(k-1)!!(n-k)!!}{(k-1-2i)!!(n-k+2i)!!}\int_{\p K} H_{k-1-2i}d\mu.
\end{eqnarray}
\end{itemize}
Here the notation $k!!$ means the product of all odd (even) integers up to  odd (even) $k$. For $k=n$, the formulas \eqref{even} and \eqref{odd} can be viewed as the Gauss-Bonnet-Chern theorem for domains in the hyperbolic space.

 From \eqref{even} and \eqref{odd}, one can see the difference between quermassintegrals $W_{2k}$ and $W_{2k+1}$. In fact, $W_{2k}$ is extrinsic
 and $W_{2k+1}$ is intrinsic, namely it depends only on the induced metric $g$ on $\partial K$.  The latter follows from the fact that $H_{2k}$ can be
 expressed in terms of intrinsic geometric quantities, the Gauss-Bonnet curvatures. For the proof see \cite{GWW2}.

\
 
\section{Quermassintegral preserving curvature flows}

Let $K_0\in \KK$ be an h-convex bounded domain in $\HH^n$ with smooth boundary $\S_0=\p K_0$.
We consider the following curvature evolution equation
\begin{eqnarray}\label{flow}
\frac{\p X}{\p t}(x,t)=(c(t)-F(\mathcal{W}(x,t)))\nu(x,t),
\end{eqnarray}
where $X(\cdot,t): M^{n-1}\to \HH^n$ are parametrizations of a family of hypersurfaces $\S_t\subset\HH^n$ which encloses $K_t$, $\nu(\cdot,t)$ is the unit outward normal to $\S_t$, $F$ is a smooth curvature function evaluated at the matrix of the Weingarten map $\mathcal{W}$ of $\S_t$. The time dependent term $c(t)$ will be explained later.

The function $F$ should have the following properties (P):
\begin{itemize}
\item $F(A)=f(\l(A))$, where $\l(A)=(\l_1,\cdots,\l_{n-1})$ are the eigenvalues of the matrix $A$ and 
$f$ is a smooth, symmetric function defined on the positive cone $$\Gamma^+=\{\l\in \RR^{n-1}|\l_i>0, \forall i=1,\cdots, n-1\};$$
\item $f$ is positively homogeneous of degree $1$: $f(t\l)=tf(\l)$ for any $t>0$;
\item $f$ is strictly increasing in each argument: $\frac{\p f}{\p \l_i}>0;$
\item $f$ is normalized by setting $f(1,\cdots,1)=1$;
\item $f$ is concave and inverse concave, i.e., ${f}^*(\l):=-f(\l_1^{-1},\cdots,\l_{n-1}^{-1})$ is concave.
\end{itemize}

We use the notation $\dot{f}^i=\frac{\p f}{\p \l_i}$, $\ddot{f}^{ij}=\frac{\p f}{\p \l_i\p \l_j}$, $F^{ij}=\frac{\p F}{\p A_{ij}}$ and $F^{ij,rs}=\frac{\p^2 F}{\p A_{ij} \p A_{rs}}$. Also we use ``$\nabla$" or ``;" to denote the covariant derivative on hypersurfaces. Unless stated otherwise, the summation convention is used throughout this paper. For our purpose, $F$ is viewed as a function on $h_i^j=g^{jk}h_{ik}$, i.e.,  $F=F(h_{i}^j)=F(g^{jk}h_{ik})=f(\kappa)$, where  $g_{ij}$ and $h_{ij}$ the first and second fundamental form respectively and $\k=(\k_1,\cdots,\k_{n-1})$ is the set of the principal curvatures.  




\

We have the evolution equations for the quermassintegrals and the curvature integrals associated with $K_t$ under  flow \eqref{flow}.
\begin{proposition} \label{pro3.1}Along  flow \eqref{flow}, we have 
\begin{eqnarray}\label{var1}
\frac{d}{dt}\hbox{Vol}(K_t)=\int_{\S_t} \left(c(t)-F\right)d\mu_t;
\end{eqnarray}
\begin{eqnarray}\label{var2}
\frac{d}{dt}|\S_t|=\int_{\S_t} (n-1)H_1\left(c(t)-F\right)d\mu_t;
\end{eqnarray}
\begin{eqnarray}\label{var3}
&&\frac{d}{dt}\int_{\S_t} H_k d\mu_t=\int_{\S_t} \left\{(n-1-k)H_{k+1}+kH_{k-1}\right\}\left(c(t)-F\right)d\mu_t, \quad k=1,\cdots, n-1;
\end{eqnarray}
\begin{eqnarray}\label{var4}
\frac{d}{dt}W_k(K_t)=\frac{n-k}{n}\int_{\S_t} H_{k}\left(c(t)-F\right)d\mu_t, \quad k=0,\cdots,n-1.
\end{eqnarray}
\end{proposition}
\begin{proof}
\eqref{var1}--\eqref{var3}  are now well-known and were proved in \cite{Reilly}. We now prove \eqref{var4} by induction. In view of \eqref{var1} and \eqref{var2}, it is true for $k=0,1$.  Assume it is true for $k-1$, we can compute by using \eqref{relation}, \eqref{var3} and the inductive assumption that
\begin{eqnarray*}
\frac{d}{dt}W_{k+1}(K_t)&=&\frac{1}{n}\frac{d}{dt}\int_{\S_t} H_{k}d\mu_t-\frac{k}{n-k+1}\frac{d}{dt}W_{k-1}(K_t)\\&=&\frac{1}{n}\int_{\S_t}\left((n-1-k)H_{k+1}+kH_{k-1}\right)\left(c(t)-F\right)d\mu_t\\&&-\frac{k}{n-k+1}\frac{n-k+1}{n}\int_{\S_t} H_{k-1}\left(c(t)-F\right)d\mu_t\\&=&
\frac{n-k-1}{n}\int_{\S_t} H_{k+1}\left(c(t)-F\right)d\mu_t.
\end{eqnarray*}
\end{proof}

The choice of $c(t)$  depends on which geometric quantity we want to preserve. In this paper, we will take 
\begin{eqnarray}\label{ct}
c(t)=c_l(t):=\frac{\int_{\S_t}H_l F d \mu_t}{\int_{\S_t}H_l d\mu_t}
\end{eqnarray}
so that the flow preserves $W_l$.

\begin{lemma} With the choice of $c(t)$ by \eqref{ct} flow \eqref{flow} preserves the quermassintegral $W_l$.
\end{lemma}
\begin{proof} By \eqref{var4} we have
\[\frac{d}{dt}W_l(K_t)=\frac {n-l}n\int_{\S_t}H_l(c_l(t)-F) d\mu_t=0.\]
\end{proof}

Under the assumptions (P) on $F$ and the assumption that the initial domain is h-convex, the long time existence and convergence of the flow \eqref{flow} can be proved.

\begin{theorem}\label{thmflow}
Let $K_0\in \KK$ be an h-convex  domain in $\HH^n$ with smooth boundary $\S_0$. Let $F$ be a function satisfying the properties (P) and $c(t)$ be defined in \eqref{ct} for some $l\in\{0,\cdots, n-1\}$. Then  flow \eqref{flow} has a smooth solution $X(t)$ for $t\in [0,\infty)$. Moreover, $X(t)$ converges exponentially to a geodesic sphere with the same quermassintegral $W_l$ as $K_0$.
\end{theorem}

\begin{proof} The proof will be divided into two steps.

\

\noindent\textbf{Step I.} The flow \eqref{flow} exists at least in a short time interval $[0,T^*)$ for some $T^*>0$ and the evolving hypersurface $\S_t$ is strictly h-convex for all $t\in (0,T^*).$

The short time existence is now well-known, since the third condition in (P) ensures that  the  flow is strictly parabolic.  
To prove the strict h-convexity, we shall use the following constant rank theorem.

\begin{theorem}\label{rank}Let $\S_t$ be a smooth solution to the flow \eqref{flow} in $[0,T]$ for some $T>0$ which is h-convex, i.e.,  the matrix $(S_{ij})=(h_{ij}-g_{ij})\geq 0$ for $t\in [0,T]$. Then $(S_{ij})$ is of constant rank $l(t)$ for each $t\in (0,T]$ and $l(s)\leq l(t)$ for all $0<s\leq t\leq T$.
\end{theorem}

\begin{proof}
The proof follows similar arguments  as that of the proof of Proposition 5.1 in \cite{BG}. 
For the convenience of the readers, we sketch the proof.

The h-convexity of $\S_0$ means that $(S_{ij})\geq 0$ at $t=0$. For $\varepsilon>0$, define a symmetric matrix $W=(S_{ij}+\varepsilon g_{ij})$. Let $l(t)$ be the minimal rank of $(S_{ij}(x,t))$. For a fixed $t_0\in (0,T]$, let $x_0\in S_{t_0}$ such that $(S_{ij}(x,t_0))$ attains its minimal rank at $x_0$. Set $\phi(x,t)=\sigma_{l+1}(W(x,t))+\frac{\sigma_{l+2}}{\sigma_{l+1}}(W(x,t)).$ It is proved in Section 2 in \cite{BG} that $\phi$ is in $C^{1,1}$. We will show that there are constants $C_1, C_2$ and $\delta$, depending on  $\|X\|_{C^{3,1}(M\times [0, T^*))}$ but independent of $\varepsilon$ and $\phi$, such that in some neighborhood $\mathcal{O}$ of $x_0$ and for $t\in (t_0-\delta,t_0]$, \begin{eqnarray}\label{strongmax1}
F^{ij}\phi_{;ij}-\frac{\p}{\p t} \phi\leq C_1\phi+C_2|\nabla \phi|.
\end{eqnarray}
As  in \cite{BG}, in $\mathcal{O}\times (t_0-\delta,t_0]$, the index set $\{1,\cdots,n-1\}$ can be divided into two subsets $B$ and $G$, where for $i\in B$, the eigenvalues $\tilde{\l}_i$ of $W$ is small and for $j\in G$, $\tilde{\l}_j$ is uniformly positive away from $0$. By choosing suitable coordinates, we may assume at each point of computation, $W_{ij}$ is diagonal.

 One can verify the evolution equation for $S_{ij}$ (see e.g. (4.23) in \cite{Gerhardt}):
\begin{eqnarray}\label{evolv}
\frac{\p}{\p t}S_{ij}&=& F^{kl} h_{ij;kl}+F^{kl,rs}h_{kl;i} h_{rs;j}+F^{kl}h_{lr}h_k^rh_{ij}-c(t)h_i^kh_{jk}\nonumber\\&&-(2F-c(t))g_{ij}+F^{kl}g_{kl}h_{ij}-2(F-c(t))h_i^kS_{jk}\nonumber\\
&=&F^{kl}S_{ij;kl}+F^{kl,rs}S_{kl;i}S_{rs;j}\nonumber\\
&&+F^{kl}h_{lr}h_k^rS_{ij}-c(t)S_i^kS_{jk}-2c(t)S_{ij}+F^{kl}g_{kl}S_{ij}-2(F-c(t))h_i^kS_{jk}\nonumber\\
&&+F^{kl}h_{lr}h_k^rg_{ij}+(F^{kl}g_{kl}-2F)g_{ij}.
\end{eqnarray}
The last line of \eqref{evolv} can be further computed as
\begin{eqnarray}\label{lastline}
&&F^{kl}h_{lr}h_k^rg_{ij}+(F^{kl}g_{kl}-2F)g_{ij}\nonumber\\&=&\left(F^{kl}S_{lr}S_k^r+2F^{kl}S_{kl}+F^{kl}g_{kl}+F^{kl}g_{kl}-2F^{kl}h_{kl}\right)g_{ij}\nonumber\\
&=&F^{kl}S_{lr}S_k^rg_{ij}\geq 0,
\end{eqnarray}
since $F^{kl}$ is positive definite and $F^{kl}h_{kl}=F$ due to the $1$-homogeneity of $f$.

Let $O(\phi)$ denote the quantity which can be controlled by $C\phi$ for a universal constant $C$ depending  on  $\|X\|_{C^{3,1}(M\times [0, T^*))}$ but independent of $\varepsilon$ and $\phi$. Notice that $\varepsilon=O(\phi)$ near $(x_0,t_0)$ (see (3.8)  in \cite{BG}). 
With help of this, we can compute by using \eqref{evolv} and \eqref{lastline} that
\begin{eqnarray}\label{strongmax}
F^{kl}\phi_{;kl}-\frac{\p}{\p t}\phi&=&\phi^{ii}(F^{kl}W_{ii;kl}-\frac{\p}{\p t} W_{ii})+F^{kl}\phi^{ij,rs}W_{ij;r}W_{kl;s}\nonumber
\\&\leq&\phi^{ii}\big(-F^{kl,rs}W_{kl;i} W_{rs;i}-F^{kl}h_{lr}h_k^rW_{ii}\nonumber
\\&&+c(t)W_i^kW_{ik}+2c(t)W_{ii}-F^{kl}g_{kl}W_{ii}+2(F-c(t))h_i^kW_{ik}\big)\nonumber
\\&&+F^{kl}\phi^{ij,rs}W_{ij;k}W_{kl;s}+O(\phi).
\end{eqnarray}
Here we use the notation $\phi^{ij}=\frac{\p \phi}{\p W_{ij}}$ and $\phi^{ij,kl}=\frac{\p^2 \phi}{\p W_{ij}\p W_{kl}}$.

Now since $W_{ij}$ satisfies Codazzi property, $\phi^{jj}=O(\phi)$ for $j\in G$ ((3.14) in \cite{BG}) and  $W_{ii}\leq \phi$ for $i\in B$,  we can use the same argument as Theorem 3.2 in \cite{BG} to reduce \eqref{strongmax} to  the following inequality as corresponding to inequality (3.19) in \cite{BG}:

\begin{eqnarray*}
&&F^{kl}\phi_{;kl}-\frac{\p}{\p t}\phi\\&\leq&-\phi^{ii}F^{kl,rs}W_{kl;i} W_{rs;i}+F^{kl}\phi^{ij,rs}W_{ij;k}W_{kl;s}+O(\phi)
\\&\leq&O(\phi+\sum_{i,j\in B}|\nabla W_{ij}|)-\frac{1}{\sigma_1(B)}\sum_{k,l}\sum_{i,j\in B,i\neq j} F^{kl}W_{ij;k}W_{ij;l}\\&&-\frac{1}{\sigma_1^3(B)}\sum_{k,l}\sum_{i\in B} F^{kl}\left(W_{ii;k}\sigma_1(B)-W_{ii}\sum_{j\in B}W_{jj;k}\right)\left(W_{ii;l}\sigma_1(B)-W_{ii}\sum_{j\in B}W_{jj;l}\right)\\&&-\sum_{i\in B}\left[\sigma_l(G)+\frac{\sigma_1^2(B|i)-\sigma_2(B|i)}{\sigma_1^2(B)}\right]\cdot\nonumber\\&&\quad\quad\quad\cdot\left[\sum_{k,l,r,s\in G}F^{kl,rs}W_{kl;i}W_{rs;i}+2\sum_{k,l\in G}F^{kl}\sum_{j\in G}\frac{1}{\tilde{\l}_j}W_{ij;k}W_{ij;l}\right].
\end{eqnarray*}
Here $\sigma_k(B)$ denotes the symmetric functions $\sigma_k$ on the eigenvalues $\tilde{\l}_i$ for $i \in B$.

The analysis  in Theorem 3.2 in \cite{BG} shows that the right hand side of above inequality can be controlled by 
$\phi+|\nabla \phi|-C\sum_{i,j\in B}|\nabla W_{ij}|$. We remark that the inverse concavity of $F$ plays an crucial role in this analysis. Hence we arrive at \eqref{strongmax1}. Now letting $\varepsilon\to 0$ and by the standard strong maximum principle for parabolic equations, we conclude that $S_{ij}$ is of constant rank $l(t)$ and $l(t)$ is non-decreasing with respect to $t$.

\end{proof}

We return to the proof of Theorem \ref{thmflow}, Step I. We follow closely the argument in \cite{BG}.

We may approximate $\S_0$ by a family of strictly h-convex hypersurfaces $\S_0^\varepsilon$. By continuity, there is $\delta>0$ (independent of $\varepsilon$), such that there is a solution $\S_t^\varepsilon$ to \eqref{flow} for $t\in [0,\delta]$. Then $\S_t^\varepsilon$ must be strictly h-convex  for  $t\in [0,\delta]$ by Theorem \ref{rank}. Taking $\varepsilon\to 0$, we have that  $\S_t$ is h-convex for  $t\in [0,\delta]$. This implies that the set $\{t\in [0,T]| \S_t \hbox{ is h-convex}\}$ is open. It is obviously closed and non-empty. Therefore, $\S_t$ is h-convex for $t\in [0,T]$ . Recall that for every closed hypersurface, there exists at least one point which is strictly h-convex. Therefore by Theorem \ref{thmflow} again, $\S_t$ is strictly h-convex for all $t\in (0,T]$. We finish the proof of Step I.

\

\noindent\textbf{Step II:} Let  $\S_{t_0}$, $t_0\in (0,T^*)$ be a strictly h-convex hypersurface evolving by \eqref{flow}, then the long time existence and convergence can be proved.

Starting with a strictly h-convex hypersurface, the flow \eqref{flow} is quite similar to that considered by Makowski \cite{Ma}. The  difference is that the flows he considered preserve the curvature integrals and ours preserve the quermassintegrals. However, this difference makes a very big difference in applications, though the analytic part of both flows
is quite similar. For the convenience of the readers, we sketch the proof and point out where  the difference is.


Let $\H^n=\RR\times\SS^{n-1}$ with the hyperbolic metric 
\begin{eqnarray*}
\bar{g}=dr^2+\sinh^2r g_{\SS^{n-1}}
\end{eqnarray*}
where $g_{\SS^{n-1}}$ is the standard round metric on the $(n-1)$-dimensional  unit sphere. 
Denote by $\<\cdot,\cdot\>$ the metric $\bar{g}$, and by $\bar{\nabla}$ 
the covariant derivative on $\H^n$. 

\

\noindent\textbf{1.} As long as the flow exists, the strict h-convexity and the pinching of the principal curvatures are preserved (Lemma 4.4 in \cite{Ma}), i.e.,
\begin{itemize}
\item if $h_{ij}-g_{ij}\geq \varepsilon g_{ij}$ at $t=t_0$  for some $\varepsilon>0$, then it holds as long as the flow exists;
\item if $h_{ij}-g_{ij}\geq \varepsilon (H_1-1) g_{ij}$ at $t=t_0$  for some $\varepsilon>0$,
 then it holds  as well as $h_{ij}-g_{ij}\geq \varepsilon (F-1) g_{ij}$ holds as long as the flow exists.
\end{itemize}
This can be proved by using Andrews' pinching estimates \cite{An}. The first statement also follows from Theorem \ref{rank}.

An important consequence of the pinching estimate is that the flow is always uniformly parabolic, i.e.,  there exists some constant $c_0$, depending only on $\S_{t_0}$, such that

\begin{eqnarray}\label{unip}
c_0^{-1}g^{ij}\leq F^{ij}\big((h_j^i)(x,t)\big)\leq c_0g^{ij}, \quad t_0\leq t< T^*.
\end{eqnarray}

\

\noindent\textbf{2.} As long as the flow exists, the speed function $F$ is bounded by a constant depending only on the initial hypersurface $\S_{t_0}$. Consequently, the time-dependent term $c(t)$ is bounded, and $|\frac{\p X}{\p t}|$ is bounded.
By the pinching estimate in Step II.1, one can easily deduce the upper boundedness of the principal curvatures. 

The proof of the boundedness of $F$ is more technique. Hence we give more details for this step.

\

\noindent{\bf 2.1}.  As long as the flow exists, the inner radius and the outer radius of $K_t$ can be uniformly bounded by some positive constants $r_0$ and $R_0$, dependent only on the initial hypersurface $\S_{t_0}$, respectively. 

In fact, this is the only place where the property of preserving the quermassintegrals is used. We verify this here.
Let $r(t)$ and $R(t)$ be the inner radius and outer radius of $\S_t$ respectively. Let $r_{t_0}$ be the number  so that $W_l(K_{t_0})=W_l(B_{r_{t_0}}).$ By virtue of \eqref{inclus}, we have that
\begin{eqnarray*}
W_l(B_{R(t)})\geq W_l(K_t)=W_l(K_{t_0})=W_l(B_{r_{t_0}}).
\end{eqnarray*}
Thus $R(t)\geq r_{t_0}$. According to Step I, the h-convexity is preserved. A remarkable feature of the h-convexity is that the inner radius and the outer radius are comparable (see \cite{Ma}, Theorem 5.2   or \cite{BM}, Theorem 3.1). Namely, there is a constant $C>1$  such that $$r(t) \le R(t) \le Cr(t).$$ Hence
$$r(t)\geq C^{-1}R(t)\geq C^{-1}r_{{t_0}}:=r_0.$$ Similarly, from the monotonicity of the quermassintegral \eqref{inclus}, we have
\[W_l(B_{r(t)})\le W_l(K_t)=W_l(K_{t_0})=W_l(B_{r_{t_0}}),\]
which implies $r(t)\le r_{t_0}$. Hence,  we have $$R(t)\le Cr(t)\le Cr_{t_0}:=R_0.$$

\

\noindent{\bf  2.2}. 
Fix a time $t_1\in [t_0,T^*)$. Since the inner radius of $K_t$ is uniformly bounded, we can assume $B_{r_{t_1}}(p_{t_1})\subset K_{t_1}$ is an enclosed ball with the center $p_{t_1}$ and the radius $r_{t_1}\geq r_0$, then we can show that $B_{\frac12r_{t_1}}(p_{t_1})\subset K_{t}$ in some short time interval $t\in [t_1, t_2)$ for $t_2$ chosen later.

In fact, let $r(x,t)$ be the distance function of $\S_t$ from $p_{t_1}$. Set $\rho(x,t):=\cosh r(x,t)$. Let $u:=\<\bar{\nabla}\rho,\nu\>$ be the ``support function". Define $$\varphi:=e^{(n-1)c_0(t-t_1)}\rho(x,t),$$ where $c_0$ is the constant in \eqref{unip}. Using the fact that $\rho_{;ij}=\rho g_{ij}-uh_{ij}$ and $F=F^{ij}h_{ij}$,  one can easily check that 
\begin{eqnarray*}
&&\frac{d}{dt}\varphi-F^{ij}\varphi_{;ij}=\varphi\left((n-1)c_0-F^{ij}g_{ij}\right)+c(t)e^{(n-1)c_0(t-t_1)}u\geq 0.\end{eqnarray*}
By parabolic maximum principle, $$\inf_{x\in \S_t} \rho(x,t)\geq e^{-(n-1)c_0(t-t_1)}\inf_{x\in \S_{t_1}} \rho(x,t_1)\geq e^{-(n-1)c_0(t-t_1)}\cosh r_{t_1}.$$
Therefore, in the time interval $[t_1, t_2)$, where $t_2=\min\{t_1+\frac{1}{(n-1)c_0}\ln\frac{ \cosh r_{t_1}}{\cosh \frac12 r_{t_1}}, T^*\}$, we have $r(x,t)\geq \frac12r_{t_1}$, namely, $B_{\frac12r_{t_1}}(p_{t_1})\subset K_{t}$.

Moreover, in view of a crucial property of h-convexity, which says $\<\p_r,\nu\>\geq \tanh r$ (see e.g. \cite{CM}, Theorem 4),  we infer that the ``support function" $u=\sinh r\<\p_r,\nu\>$
  is bounded below by a positive constant $u_0:=\sinh \frac12 r_{t_1}\tanh \frac12 r_{t_1}$ in the time interval $[t_1, t_2)$.  On the other hand, h-convexity ensures that $r(x,t)\leq r(t)+\ln 2\leq R_0+\ln2$ (see e.g. \cite{CM}, Theorem 4), which implies that $u$ is also bounded above.

\

\noindent{\bf 2.3}. 
In the time interval $[t_1, t_2)$, we consider an auxiliary function $$\Phi:=\frac{F}{u-\frac12{u_0}}.$$ One can verify the evolution equation of $\Phi$:

\begin{eqnarray}\label{eqf}
\frac{d}{dt}\Phi&=&F^{ij}\Phi_{;ij}+\frac{2F^{ij}u_{;i}\Phi_{;j}}{u-\frac12{u_0}}-\frac{c(t)}{u-\frac12{u_0}}(F^{ij}h_i^kh_{kj}-F^{ij}g_{ij}
)\nonumber\\&&-\frac{\frac12{u_0}}{(u-\frac12{u_0})^2}F^{ij}h_i^kh_{kj}F+\frac{2F-c(t)}{u-\frac12u_0}\cosh r\Phi-\frac{F}{u-\frac12u_0}F^{ij}g_{ij}.
\end{eqnarray}

By the h-convexity of $\S_t$, we know $F^{ij}h_i^kh_{kj}-F^{ij}g_{ij}\geq 0$. Also, by pinching estimate, we have
$$F^{ij}h_i^kh_{kj}\geq \varepsilon F^2.$$
Hence at the maximum point of $\Phi$ in $M\times [t_1,t_2)$, we deduce from \eqref{eqf} that
\begin{eqnarray}\label{eqf1}
0\leq \frac{d}{dt}\Phi\leq 
-\frac12{u_0}(u-\frac12{u_0})\varepsilon \Phi^3+2\cosh r\Phi^2.
\end{eqnarray}
Since $u-\frac12u_0\geq \frac12u_0$ and $r\leq R_0+\ln 2$, it follows from \eqref{eqf1} that  for  $t\in [t_1,t_2)$, $\Phi$ is bounded above by a constant $C$ depending only on $\S_{t_1}$. Consequently, as $u$ has also upper bound, we get that the speed $F$ is bounded above by $C$ for $t\in [t_1,t_2)$. Since $t_1$ can be chosen arbitrary in $[t_0,T^*)$, we conclude that $F$ has a uniform bound for  $t\in [t_0,T^*)$.

\

\noindent\textbf{3.} The flow exists for $t\in [0,\infty)$ and the flow convergence to a geodesic sphere.

In view of Step II.2,  the $C^2$ estimate for the graph function $r(x,t)$ is available. On the other hand, one can obtain a positive lower bound for $F$ by the  parabolic Harnack inequality (Lemma 6.2 \cite{Ma}). This, combining with the pinching estimate, yields the positive lower bounds for the principal curvatures. Hence, we conclude that the principal curvatures lie in a compact set of $\Gamma^+$. Taking into account that the flow is uniformly parabolic (Step II.1), we can derive the higher order estimates exactly as in \cite{Mc}, Section 8. Finally, we prove the long time existence in a standard way. 
The flow convergence to a geodesic sphere is proved by showing that the pinching of the principal curvatures is improving at an exponential rate (Proposition 7.1 and Corollary 7.2 in \cite{Ma}).

\end{proof}

\begin{remark}
 With Theorem \ref{rank}, one can show that the results in \cite{Ma} hold for h-convex hypersurfaces.
\end{remark}

\

A direct consequence of Theorem \ref{thmflow} is the following Alexandrov type theorem for hypersurfaces in $\H^n$.

\begin{corollary}\label{rig}
Let $0\leq l<k\leq n-1$.  Let $K\in \KK$ be an   h-convex  bounded domain in $\HH^n$ with smooth boundary satisfying that  $H_k=cH_l$ for some constant $c\in\RR$. Then $K$ must be a geodesic ball.
\end{corollary}
\begin{proof}
We just let $K$ evolve by \eqref{flow} with $F=\left(\frac{H_k}{H_l}\right)^\frac{1}{k-l}$. Then the flow is actually stationary. The convergence of Theorem \ref{thmflow} implies  that  $K$ must be a geodesic ball.
\end{proof}

Here we  provide another direct proof for the rigidity, which is applicable for a wide class of domains.

\begin{theorem}\label{rigidity}
Let $0\leq l<k\leq n-1$.  Let $K$ be a bounded domain in $\HH^n$ with smooth $k$-convex boundary $\Sigma$, namely, the principal curvatures of $\Sigma$ lie in $\overline{\Gamma_k^+}$.  If $H_k=cH_l$ for some constant $c\in\RR$, then $K$ must be a geodesic ball.
\end{theorem}

\begin{proof} This can be proved by using the maximum principle, as mentioned in the paper of Korevaar \cite{Ko}. 
For convenience of the reader, we give a simple proof in the spirit of the work of Montiel-Ros \cite{MoRos} and \cite{Ros}.

Let $\H^n=\RR\times\SS^{n-1}$ with the hyperbolic metric 
\begin{eqnarray*}
\bar{g}=dr^2+\sinh^2r g_{\SS^{n-1}}
\end{eqnarray*}
as above.
Recall the function $$\rho:\H^n\to \RR: \quad \rho(x)=\cosh r(x).$$ 
The Minkowski formula in $\HH^n$ (see \cite{MoRos}, Lemma 2) tells that 
\begin{eqnarray}\label{Mink}
\int_{\Sigma} H_k\<\bar{\nabla}\rho,\nu\>d\mu=\int_{\Sigma} H_{k-1}\rho d\mu,\quad k=1,\cdots,n-1.
\end{eqnarray}

We consider two cases separately. 

For the case $l=0$, we have $H_k=c$. As explained in \cite{Ros}, this implies that $c$ is a positive constant and $H_j$ is positive for $1\leq j\leq k$. A Heintze-Karcher type inequality, recently proved by Brendle \cite{Brendle}, says that 
\begin{eqnarray}\label{HK}
\int_{\Sigma} \<\bar{\nabla}\rho,\nu\>d\mu
\leq  \int_{\Sigma} \frac{1}{H_1}\rho d\mu,\end{eqnarray}
and equality holds if and only if $\Sigma$ is a geodesic sphere.

By  \eqref{Mink}, \eqref{HK} and  the Newton-Maclaurin inequality \eqref{N1}, we have
\begin{eqnarray}\label{int}
\int_{\Sigma} H_{k-1}\rho d\mu&=& c\int_{\Sigma} \<\bar{\nabla}\rho,\nu\>d\mu
\leq  c\int_{\Sigma} \frac{1}{H_1}\rho d\mu\nonumber\\ &\leq & c\int_{\Sigma} \frac{H_{k-1}}{H_{k}}\rho d\mu
=\int_{\Sigma} H_{k-1}\rho d\mu.
\end{eqnarray}
Now equality holds in \eqref{HK} and \eqref{N1}, whence $\Sigma$ is a geodesic sphere.

For $l\geq 1$, we use \eqref{Mink} twice and the assumption $H_k=cH_l$ to get 
\begin{eqnarray*}
\int_{\Sigma} H_{k-1}\rho d\mu=\int_{\Sigma} H_k\<\bar{\nabla}\rho,\nu\>d\mu=c\int_{\Sigma} H_l\<\bar{\nabla}\rho,\nu\>d\mu=c\int_{\Sigma} H_{l-1}\rho d\mu.
\end{eqnarray*}
Therefore,
\begin{eqnarray}\label{rig2}
\int_{\Sigma} (H_{k-1}- cH_{l-1})\rho d\mu=0.
\end{eqnarray}
On the other hand, by the Newton-Maclaurin's inequality \eqref{N1}, we have
\begin{eqnarray}\label{rig1}
H_{k-1}- cH_{l-1}=H_{k-1}- \frac{H_k}{H_l}H_{l-1}\geq 0.
\end{eqnarray}
at the points where $H_l\neq 0$. For the points where $H_k=H_l=0$, we have
by the Newton-Maclaurin's inequality \eqref{N2} that 
$$0=H_l\geq H_k^{\frac{l}{k}}=0,$$
which implies all the principal curvatures at those points are zero. Hence \eqref{rig1} still holds at the points where $H_k=H_l=0$.
Thus, one immediately see from \eqref{rig1} and \eqref{rig2} that equality holds in \eqref{rig1}, which implies that $\Sigma$ is a geodesic sphere.
\end{proof}

\

\section{Proof of  theorems}

Before proving the theorems,  we define some auxiliary functions which will be used below.

First recall that, for $0\leq k\leq n-1$,
\begin{eqnarray*}
f_k:[0,\infty)\to \RR_+, \quad f_k(r)=W_k(B_r).\end{eqnarray*}
  It is easy to see that $f$ is smooth and it follows from \eqref{inclus} that $f_k$ is strictly monotone increasing.
Hence its inverse function $f_k^{-1}$ exists and is also strictly monotone increasing.

For $2\leq k\leq n-1$, define
\begin{eqnarray*}
g_{k}: [0,\infty)\to\RR_+, &&\quad g_k(s)=nf_{k}\circ f_{k-2}^{-1}(s)+\frac{n(k-1)}{n-k+2}s.\end{eqnarray*}
Thanks to the monotonicity of $f_k$, $g_k$ is also strictly monotone increasing and 
 its inverse function $g_k^{-1}$ exists and  is strictly monotone increasing. One can easily check from the definition of $f_k$ and $g_k$ that \begin{eqnarray}\label{pos}
\frac1ns-\frac{k-1}{n-k+2}g_{k}^{-1}(s)\geq 0.
\end{eqnarray}

For $1\leq k\leq n-1$, define
\begin{eqnarray*}
h_{k}:[0,\infty)\to\mathbb{R}_+,&&\quad h_{k}(s)=g_{k+1}\left(\frac1ns-\frac{k-2}{n-k+3}g_{k-1}^{-1}(s)\right).
\end{eqnarray*}
We claim that $h_k$ is also strictly monotone increasing. Indeed, it is direct to compute that
\begin{eqnarray*}
&&h_{k}'(s)= g_{k+1}'\left(\frac1ns-\frac{k-2}{n-k+3}g_{k-1}^{-1}(s)\right)\cdot\left(\frac1n-\frac{k-2}{n-k+3}\frac{1}{g_{k-1}'(g_{k-1}^{-1}(s))}\right)
\end{eqnarray*}
Since $g_{k+1}'>0$ and
$$g_{k-1}'=n(f_{k-1}\circ f_{k-3})'+n\frac{k-2}{n-k+3}>n\frac{k-2}{n-k+3},$$
we have that $h_{k}'>0$, namely  $h_{k}$ is strictly monotone increasing.


\

Now we start to prove main theorems.
We first prove Theorem \ref{thm1} by using special forms of  flow \eqref{flow}.

\

\noindent{\it Proof of Theorem \ref{thm1}.}
Let $K=K_0\in \KK$.

\

To prove Theorem \ref{thm1}, we consider  flow \eqref{flow} starting from $\S_0=\p K_0$ with $$F=\left(\frac{H_k}{H_l}\right)^{\frac{1}{k-l}},\quad c(t)=c_l(t)=\frac{\int_{\S_t} H_k^{\frac{1}{k-l}}H_l^{1-\frac{1}{k-l}}d\mu_t}{\int_{\S_t} H_l d\mu_t}.$$
Let $\S_t, t\in [0,\infty)$ be the solution in obtained Theorem \ref{thmflow}, which encloses $K_t$. One verifies from \eqref{var4} that
\begin{equation}\begin{array}{rcl}\label{varr}
\ds\vs \frac{d}{dt} W_k(K_t)&=&\ds \int_{\S_t} H_k \left(c(t)-F\right)\\
&=& \ds \frac{n-k}{n}\frac{1}{\int_{\S_t} H_l }\left(\int_{\S_t} H_k \int_{\S_t} H_k^{\frac{1}{k-l}}H_l^{1-\frac{1}{k-l}}-\int_{\S_t} H_l \int_{\S_t} H_k^{1+\frac{1}{k-l}}H_l^{-\frac{1}{k-l}}\right)
\end{array}
\end{equation}
It follows from the H\"older inequality that
\begin{eqnarray}\label{holder1}
\int_{\S_t} H_k  \leq \left( \int_{\S_t} H_k^{1+\frac{1}{k-l}}H_l^{-\frac{1}{k-l}}\right)^\frac{k-l}{k-l+1}\left(\int_{\S_t} H_l \right)^{\frac{1}{k-l+1}},
\end{eqnarray}
\begin{eqnarray}\label{holder2}
\int_{\S_t} H_k^{\frac{1}{k-l}}H_l^{1-\frac{1}{k-l}}\leq \left( \int_{\S_t} H_k^{1+\frac{1}{k-l}}H_l^{-\frac{1}{k-l}}\right)^\frac{1}{k-l+1}\left(\int_{\S_t} H_l \right)^{\frac{k-l}{k-l+1}}.
\end{eqnarray}
Inserting \eqref{holder1} and \eqref{holder2} into \eqref{varr}, we have 
\begin{eqnarray}\label{dec}
\frac{d}{dt} W_k(K_t)\leq 0.
\end{eqnarray}
Note that the flow preserves $W_l$. Theorem \ref{thmflow} says that the flow converges to some geodesic ball $B_r$ with $W_l(B_r)=W_l(K_0)=W_l(K_t)$.
Thus we have 
\begin{eqnarray}\label{decr}
W_k(K)\ge W_k(B_r),\quad \hbox{ with } W_l(K)=W_l(B_r)\hbox{ for some }r>0,
\end{eqnarray}
which is equivalent to
\begin{eqnarray}\label{ineq1}
W_k(K)\geq f_k\circ f_l^{-1}(W_l(K)).
\end{eqnarray}
Equality in \eqref{ineq1} holds iff equalities in \eqref{holder1} and \eqref{holder2} hold, iff $H_k=cH_l$ for some $c\in \RR$, which means by Corollary \ref{rig} or
 Theorem \ref{rigidity} that $K$ is a geodesic ball in $\HH^n$.  
\qed 

\

\noindent{\it Proof of Theorem \ref{thm2}.}
Once we have Theorem \ref{thm1} and especailly have \eqref{ineq1}, it is easy to see from \eqref{relation} that\begin{eqnarray}\label{kl}
V_{n-1-k}(K)&=& n\left( W_{k+1}(K)+\frac{k}{n-k+1}W_{k-1}(K)\right)\nonumber\\
&\geq & \left(nf_{k+1}\circ f_{k-1}^{-1}+\frac{nk}{n-k+1}Id\right)\left(W_{k-1}(K)\right)\nonumber\\
&\geq &\left(nf_{k+1}\circ f_{l}^{-1}+\frac{nk}{n-k+1}f_{k-1}\circ f_l^{-1}\right)\left(W_{l}(K)\right),
\end{eqnarray}
where $Id:\RR\to\RR$ is the identity function. This leads to  Statement (i) in Theorem \ref{thm2}.

\

Statements (ii) and (iii) in Theorem \ref{thm2} are almost included in Statement (i) except that
(a) the area $V_{n-1}$ attains its minimum at a geodesic ball among the domains with given volume $W_0$, and (b) $\int_{\p K} H_1d\mu$ attains its minimum at a geodesic ball among the domains with given area of the boundary $|\p K|$. However, 
(a) is just the classical isoperimetric inequality in $\H^n$ and (b) was proved in \cite{GWW2} by using  results of Cheng-Zhou \cite {CZ} and Li-Wei-Xiong \cite{LWX},
which was mentioned in the introduction.


\

We now prove Statement (iv) of Theorem \ref{thm2}. First we consider the simple case $k-l=2$. For $l=0,k=2$, the statement is included in Statement (iii). Hence we assume $k\geq 3$.

First of all, we see from \eqref{relation} and \eqref{ineq1} that
\begin{eqnarray}\label{2i-2}
\int_{\p K} H_{k-2} d\mu &=&  n W_{k-1}(K)+\frac{n(k-2)}{n-k+3}W_{k-3}(K)\nonumber\\
&\geq & \left(nf_{k-1}\circ f_{k-3}^{-1}+\frac{n{k-2}}{n-k+3}Id\right) (W_{k-3}(K))\nonumber\\&=&g_{k-1} (W_{k-3}(K)).
\end{eqnarray}
It follows from \eqref{2i-2} that
\begin{eqnarray}\label{2i-3}
W_{k-3}(K)\leq g_{k-1}^{-1}\left( \int_{\p K} H_{k-2} d\mu \right).
\end{eqnarray}


Next, we use \eqref{relation} and \eqref{ineq1}  again  on $\int_{\p K} H_{k} d\mu$ to obtain that\begin{eqnarray}\label{2i}
\int_{\p K} H_{k} d\mu&\geq & g_{k+1}(W_{k-1}(K))\nonumber\\
&=&g_{k+1}\left(\frac1n\int_{\p K} H_{k-2} d\mu-\frac{k-2}{n-k+3}W_{k-3}(K)\right).
\end{eqnarray}
In view of \eqref{pos}, we deduce from \eqref{2i-3} that
\begin{eqnarray*}
&&\frac1n\int_{\p K} H_{k-2} d\mu-\frac{k-2}{n-k+3}W_{k-3}(K)\nonumber\\&\geq &\frac1n\int_{\p K} H_{k-2} d\mu-\frac{k-2}{n-k+3}g_{k-1}^{-1}\left( \int_{\p K} H_{k-2} d\mu\right)\geq 0.
\end{eqnarray*}
Back to \eqref{2i}, using the monotonicity of $g_{k+1}$, we obtain that
\begin{eqnarray}\label{last}
\int_{\p K} H_{k} d\mu &\geq& g_{k+1}\left[\frac1n\int_{\p K} H_{k-2} d\mu-\frac{k-2}{n-k+3}g_{k-1}^{-1}\left( \int_{\p K} H_{k-2} d\mu\right)\right]\nonumber\\&=&h_k\left(\int_{\p K} H_{k-2} d\mu\right).
\end{eqnarray}

For $k-l=2m$ for $m\in \NN$, due to the monotonicity of $h_k$, we can inductively utilize \eqref{last} to deduce that \begin{eqnarray}\label{last2}
\int_{\p K} H_{k} d\mu\geq h_{k}\circ h_{k-2}\circ\cdots\circ h_{l+2} \left(\int_{\p K} H_{l} d\mu\right).
\end{eqnarray}

Notice that the inequalities we have used previously are all optimal in the sense that equalities hold iff $K$ is a geodesic ball. Hence we conclude Statement  (iv) in Theorem \ref{thm2}. 

We complete the proof of Theorem \ref{thm2}. \qed

\

\

\noindent{\it Proof of Theorem \ref{thm3}}: it is sufficient to explicitly write out formula \eqref{kl} for $l=1$ and $1\leq k\leq n-1$. A direct calculation yields that 
\begin{eqnarray*}
&&f_1(r)=\frac1n|\p B_r|=\frac1n\o_{n-1}\sinh^{n-1}(r).
\end{eqnarray*}
Thus \begin{eqnarray*}
&&f_1^{-1}(s)=\sinh^{-1}\left[\left(\frac{ns}{\o_{n-1}}\right)^{\frac{1}{n-1}}\right].
\end{eqnarray*}
Since $H_k(B_r)=\coth^k(r)$, it follows from \eqref{even} and \eqref{odd} that if $k$ is odd,
\begin{eqnarray*}
&&f_k(r)=\frac1n \sum_{i=0}^{\frac{k-1}{2}} (-1)^i\frac{(k-1)!!(n-k)!!}{(k-1-2i)!!(n-k+2i)!!}\o_{n-1}\coth^{k-1-2i}(r)\sinh^{n-1}(r),
\end{eqnarray*}
while if $k$ is even,
\begin{eqnarray*}
f_k(r)&=&\frac1n \sum_{i=0}^{\frac{k}{2}-1} (-1)^i\frac{(k-1)!!(n-k)!!}{(k-1-2i)!!(n-k+2i)!!}\o_{n-1}\coth^{k-1-2i}(r)\sinh^{n-1}(r)\\&&+ (-1)^{\frac{k}{2}}\frac{(k-1)!!(n-k)!!}{n!!}\int_0^r \o_{n-1}\sinh^{n-1}(t)dt.
\end{eqnarray*}
Hence, for $k$ odd, \begin{eqnarray*}
&&f_k\circ f_1^{-1}(s)=\frac1n \sum_{i=0}^{\frac{k-1}{2}} (-1)^i\frac{(k-1)!!(n-k)!!}{(k-1-2i)!!(n-k+2i)!!}
\o_{n-1}\frac{ns}{\o_{n-1}}\left[1+\left(\frac{ns}{\o_{n-1}}\right)^{\frac{-2}{n-1}}\right]^{\frac{k-1}{2}-i},\end{eqnarray*}
and for $k$ even,
\begin{eqnarray*}
f_k\circ f_1^{-1}(s)&=&\frac1n \sum_{i=0}^{\frac{k}{2}-1} (-1)^i\frac{(k-1)!!(n-k)!!}{(k-1-2i)!!(n-k+2i)!!}\o_{n-1}\frac{ns}{\o_{n-1}}\left[1+\left(\frac{ns}{\o_{n-1}}\right)^{\frac{-2}{n-1}}\right]^\frac{k-1-2i}{2}\nonumber\\&&+ (-1)^{\frac{k}{2}}\frac{(k-1)!!(n-k)!!}{n!!}\int_0^{\sinh^{-1}\left[\left(\frac{ns}{\o_{n-1}}\right)^{\frac{1}{n-1}}\right]} \o_{n-1}\sinh^{n-1}(t)dt.
\end{eqnarray*}
From the previous two formulas  we can easily compute that for $k\geq 2$,
\begin{eqnarray}\label{last3}
nf_{k+1}\circ f_1^{-1}(s)+\frac{nk}{n-k+1}f_{k-1}\circ f_1^{-1}(s)\nonumber &=&\o_{n-1}\frac{ns}{\o_{n-1}}\left[1+\left(\frac{ns}{\o_{n-1}}\right)^{\frac{-2}{n-1}}\right]^{\frac{k}{2}}\nonumber\\&=& \omega_{n-1}\left[\left(\frac{ns}{\omega_{n-1}}\right)^{\frac{2}{k}}+\left(\frac{ns}{\omega_{n-1}}\right)^{\frac{2}{k}\frac{(n-k-1)}{n-1}}\right]^{\frac {k}{2}}.
\end{eqnarray}
Letting $s=W_1(K)=\frac1n|\p K|$ in \eqref{last3},  we obtain from \eqref{kl}  inequality \eqref{eq_add}.
\qed

\


From the proof, one can see again the difference between the even case and the odd case. However, an interesting cancellation gives
the uniform inequality \eqref{eq_add}.

\

\

\noindent{\it Acknowledgment}.
We learned from
Pengfei Guan that he and Junfang Li
obtained the same type results with the same functionals
and a different flow, a modified inverse curvature flow,
under  weaker conditions that $K$
is a $k$-convex star-shaped domain and  satisfies a
technical condition that  $|\bar{\nabla} \log \cosh r| \le C(n),$ 
where $C(n)\ge 4$ is a dimensional constant. This technical
condition is believed removable. We would like to thank Yuxin Ge,  Pengfei Guan and Jie Wu for helpful discussions.

\end{document}